\documentclass [12pt,a4paper]{amsart}
\usepackage{amsfonts}
\usepackage{amsthm}
\usepackage{amsmath, pb-diagram}
\usepackage{amscd}
\usepackage[latin2]{inputenc}
\usepackage{t1enc}
\usepackage[mathscr]{eucal}
\usepackage{indentfirst}
\usepackage{graphicx}
\usepackage{graphics,epsfig}
\usepackage{pict2e}
\usepackage{epic}
\numberwithin{equation}{section}
\usepackage[margin=2.9cm]{geometry}
\usepackage{epstopdf}
\usepackage{xfrac}
\usepackage[nice]{nicefrac}
\usepackage{float}
\usepackage[shortlabels]{enumitem}
\usepackage{mathtools}
\usepackage{amsmath,gauss,tikz}
\usepackage{lscape}
\usepackage{array}
\usepackage{setspace}
\usepackage{enumitem}

\newcommand{\digs}{\text{d-}S_{\bar{n}}}
\newcommand{\digp}{\text{d-}P_{\bar{n}}}
\newcommand{\dign}{\text{d-}P^{n_1}}
\newcommand{\digtc}{\text{d-}TC}
\newcommand{\dcat}{\text{d-cat}}

\allowdisplaybreaks

\theoremstyle{plain}
\newtheorem{theo}{Theorem}[section]
\newtheorem{cor}[theo]{Corollary}

\newtheorem{prop}[theo]{Proposition}

\theoremstyle{definition}
\newtheorem{defn}[theo]{Definition}
\newtheorem{ex}[theo]{Example}

\begin{document}
	
\title{Explicit Motion Planning in Digital Projective Product Spaces}

\author{Seher F\.{i}\c{s}ekc\.{i}}
\address{Ege University, Science Faculty, C Block, Department of Mathematics, 35100 Bornova, Izmir, Turkey}
\author{\.{I}smet Karaca}
\address{Ege University, Science Faculty, C Block, Department of Mathematics, 35100 Bornova, Izmir, Turkey}

\maketitle 

\begin{abstract}
	We introduce the digital projective product spaces based on Davis' projective product spaces. We determine an upper bound for the digital LS-category of the digital projective product spaces. In addition, we obtain an upper bound for the digital topological complexity of these spaces. We prove the relation between the digital topological complexity of the digital projective product spaces and sum of the digital topological complexity of the digital projective space by associating with the first digital sphere and the digital topological complexity of the remaining digital spheres through an explicit motion planning construction, which shows digital perspective validity of the results given by S. F\.{i}\c{s}ekc\.{i} and L. Vandembroucq. We apply our outcomes on specific spaces in order to be more clear. 

\end{abstract}

\medskip



\section{Introduction}
Topological robotics has emerged as a new mathematical discipline, having being inspired by robotics and engineering. The discipline is devoted to study by means of diverse algebraic material and methods for the concept of configuration spaces, motion planning, and the topological complexity. A configuration space is given mechanical location which describes the configurations as desired. The motion planning algorithm determines the rule of a continuous motion in the system of the given initial and final positions. The motion planning algorithm should have instability, which arises from topological reasons. The notion of the topological complexity has been introduced by M. Farber in 2003 ~\cite{Far} in order to inform topological measures of the complexity of the motion planning problem in robotics. In other words, there exists discontinuity in the motion planner on the configuration space $X$. The tool that measures this amount is the topological complexity, $TC(X)$, of the space $X$. This is a numerical homotopy invariant that can be difficult to determine. Particularly, computing the topological complexity of the n-dimensional real projective space is shown to be linked to the difficult classical well-known problem of determining the Euclidian space of minimal dimension in which this projective space can be immersed ~\cite{FTY}. 

In recent years, digital topology has played an active role in the field of topological robotics. Karaca and Is ~\cite{KI} define the concept of digital topological complexity in 2018. The notion of the digital higher topological complexity is added to the literature in ~\cite{IK}. As shown in ~\cite{IK} the cohomological lower bound, particularly zero-divisor cup-length property is not valid for the digital topological complexity. The work on digital topology in finite digital image and given counter examples underline the differences between digital topological complexity and Farber's topological complexity~\cite{IK1, IK2}. The study by Is and Karaca displays that there exists another way to state the digital topological complexity by using digital functions~\cite{IK3}.

Since the topological complexity and its related invariants are homotopy invariants, the definition and properties of digital homotopy have gained importance and some features of digital homotopy have been generalized in ~\cite{LOS}. For the Lusternik-Schnirelmann category, one of the most important related invariants of the topological complexity, the digital LS category is defined in ~\cite{BV} and the study is expanded by applying it to digital functions ~\cite{VB}. Moreover, Ege and Karaca define cohomological operation precisely cup-length in digital setting and prove that deficiency of Künneth formula in this perspective ~\cite{EK2}. According to this, the cohomological lower bound cup-length property is not to be used for the digital LS-category. We refer to studies in ~\cite{Boxer2, EK1, EK2} for more knowledge about digital topological infrastructure.

Projective product space has been introduced by Davis~\cite{Davis} in 2010. This space can be considered as a generalization of real projective space but is not in general product of projective spaces. The topological complexity and some bounds of these spaces have been initiated in ~\cite{GGTX}. The improvement of this study to finalize the estimating problem about the topological complexity and the Lusternik-Schnirelmann category of projective product spaces has been included in ~\cite{FV}. Fi\c{s}ekci and Vandembroucq compute the Lusternik-Schnirelmann category of PPS and determine an exact value of the topological complexity for some cases. This duration leads us to build the digital structure of projective product spaces and deal with the digital topological complexity and the digital LS-category of these spaces with the allowance of direct approach thanks to digital nature which contains a discrete or combinatorial sense. 

This paper relates to topological robotics, more precisely the topological complexity and most closely related invariant LS-category, is organized by starting with primary notions and basic facts in the digital frame that included the use of the analysis of the geometrical and algebraic fundamentals with digital topology. We introduce the digital projective product spaces based on Davis' projective product space by applying the digital topological tools and by using digital spheres in ~\cite{Evako}. In the process, we present digital projective spaces. Moreover, we define the digital non-singular map, and we calculate the digital topological complexity of the digital projective spaces with the digital non-singular map characterization inspired by~\cite{FTY}. We obtain new results on digital topological complexity and digital LS-category of digital projective product spaces estimating the digital topological complexity invariants through making an explicit motion planning on digital spheres. In this way, for the first time in literature, we procure that digital topological complexity and digital LS-category invariants of special spaces. We determine an upper bound of digital LS-category and consequently this yields an upper bound of digital topological complexity for these spaces. Additionally, we prove the existence of the relation between the digital topological complexity of the digital PPS and sum of the digital topological complexity of the digital projective space by considering the first digital sphere and the digital topological complexity of the remaining digital spheres. We give examples for our main results to exhibit the application on specific spaces.

\section{Preliminaries}
In this section, we give significant definitions, essential facts, useful notations for the digital topology and topological robotics.

Given any finite subset $X$ of $\mathbb{Z}^{n}$ which consists of integer points of n-dimensional Euclidean space $\mathbb{R}^{n}$. Then $(X,\kappa)$ is called a digital image ~\cite{Boxer}, where $\kappa$ is an adjacency relation on elements of $X$. $x$ and $y$ distinct points in $\mathbb{Z}^{n}$ are \textit{digital $c_{k}-$adjacent} ~\cite{Boxer} with the properties that exist at most $k$ indices $i$ such that $|x_{i} - y_{i}| = 1$ and for all remain indices $i$ such that $|x_{i} - y_{i}| \neq 1$, $x_{i} = y_{i}$, where $k$ is less than or equal to $n$. This structure provides us $c_{1} = 2-$adjacency in $\mathbb{Z}$, $c_{1} = 4-$ and $c_{2} = 8-$adjacencies in $\mathbb{Z}^{2}$, and $c_{1} = 6-$, $c_{2} = 18-$ and $c_{3} = 26-$adjacencies in $\mathbb{Z}^{3}$. Let $(X_1, \kappa_{1})$ and $(X_2, \kappa_{2})$ be any two digital image such that the points $(x, y),(x_1, y_1)$ belong to $X_1 \times X_2$. Then $(x_2, y_2)$ and $(x_1, y_1)$ are adjacent in cartesian product digital images $X_1 \times X_2$ ~\cite{Berge, Han} if one of the following features hold:
\begin{itemize}
\item $x_2 = x_1$ and $y$ and $y_1$ are $\kappa_2$ adjacent; or
\end{itemize}
\begin{itemize}
\item $x$ and $x_1$ are $\kappa_1$ adjacent and $y=y_1$; or
\end{itemize}
\begin{itemize}
\item $x$ and $x_1$ are $\kappa_1$ adjacent and $y=y_1$ are $\kappa_2$ adjacent.
\end{itemize}

Let $(X, \kappa)$ be any digital image in $\mathbb{Z}^{n}$. $X$ is called \textit{digital $\kappa$-connected} with necessary and sufficient condition that for every pair of points $x,y \in X$ with $x \neq y$, there exists $\{x_{0},x_{1}, ...,x_{l}\} \subset X$ such that $x=x_{0}$, $y=x_{l}$, $x_{i}$ and $x_{i+1}$ are $\kappa$-adjacent, where $i=0,1,...,l-1$~\cite{Herman}. Given subsets $X_1 \subset \mathbb{Z}^{n_1}$ and $X_{2} \subset \mathbb{Z}^{n_2}$, a digital map $f : (X_1, \kappa_{1}) \to (X_2, \kappa_2)$ is \textit{digital $(\kappa_{1},\kappa_{2})-$continuous} if for any digital $\kappa_{1}-$connected subset $U_{1}$ of $X_{1}$, $f(U_{1})$ is digital $\kappa_{2}-$connected~\cite{Boxer}. Furthermore, $f$ is called a \textit{digital $(\kappa_{1},\kappa_{2})$-isomorphism} if $f$ is digital $(\kappa_{1},\kappa_{2})$-continuous, bijective, and the inverse $f^{-1}$ is digital $(\kappa_{2},\kappa_{1})-$continuous~\cite{Boxer3, Han}.

A \textit{digital interval} is defined as a set $[a,b]_{\mathbb{Z}} = \{z \in \mathbb{Z} : a \leq z \leq b\}$ from $a$ to $b$ points~\cite{Boxer4}. Since $[a,b]_{\mathbb{Z}} \subset \mathbb{Z}$, it has $2$-adjacency. The notation in ~\cite{LOS} $I_m$ represents the digital interval such that $I_m \subset \mathbb{Z}$ includes integers from $0$ to $m$ in $\mathbb{Z}$, and integers are consecutively adjacent. A digital path $f$ in $X$ from $x$ to $y$ is defined by a digital map $f: I_m\to X$ is digital $(2,\kappa)$-continuous with $f(0)=x$ and $f(m)=y$~\cite{Boxer4}. The digital path $f$ is called a digital $\kappa$-loop when $f(0) = f(m)$~\cite{Boxer4}. Let $f: I_m \to X$ and  $g: I_n \to X$ be digital $\kappa$-paths with $f(m) = g(0)$. The product of these two digital paths is defined in~\cite{Kha} as the map $(f \ast g): I_{m +n} \to X$ by

\begin{displaymath}
(f \ast g)(t) =
\begin{cases}
f(t), &0 \leq t \leq m \\ 
g(t-m), &m \leq t \leq m + n.
\end{cases}
\end{displaymath}

Let $(X, \kappa)$ and $(Y, \kappa_2)$ be any two digital images. Digital $(\kappa_{1},\kappa_{2})-$continuous maps $f,g : X \to Y$ are called \textit{digitally $(\kappa_{1},\kappa_{2})-$homotopic} in $Y$ ~\cite{Boxer, Kha} if there exists $m \in \mathbb{Z}^+$ and a digital map $H : X \times I_m \to Y$ such that satisfies the following features:
\begin{itemize}
	\item for all $x \in X$, $H(x,0)=f(x)$ and $H(x,m) = g(x)$;
\end{itemize}	

\begin{itemize}	
	\item for all $x \in X$, $H_{x} : I_m \to Y$, defined by $H_{x}(t) = H(x,t)$, is digital \\ $(2,\kappa_{2})-$continuous, for all $t \in I_m$;
\end{itemize}	
	
\begin{itemize}	
	\item for all $t \in I_m$, $H_{t} : X \to Y$, defined by $H_{t}(x) = H(x,t)$, is digital \\ $(\kappa_{1},\kappa_{2})-$continuous, for all $x \in X$.
\end{itemize} 

A digital continuous map $f: X \to Y$ is \textit{digitally nullhomotopic} in $Y$ with the case that $f$ is digitally nullhomotopic to a constant map in $Y$~\cite{Boxer, Kha}. A $(\kappa_{1},\kappa_{2})-$continuous map $f: X \to Y$ is \textit{digital $(\kappa_{1},\kappa_{2})-$homotopy equivalent} to a digital $(\kappa_{2},\kappa_{1})-$continuous map $g: Y \to X$ such that $g \circ f$ is digital $(\kappa_{1},\kappa_{1})-$homotopic to the digital identity map on $X$ and $f \circ g$ is digital $(\kappa_{2},\kappa_{2})-$homotopic to the digital identity map on $Y$ ~\cite{Boxer5, Han1}. A digital image $X$ is called \textit{digital $\kappa-$ contractible} if the identity map in $X$ is digital $(\kappa, \kappa)-$ homotopic to a constant map in $X$~\cite{Boxer, Kha}.

Let $X^{I_m}$ represents the set of all digital continuous paths $\alpha : I_m \to X$ in $X$. $\pi : X^{I_m} \rightarrow X \times X$ is a digital continuous map that assigns any digital continuous paths $\alpha$ in $X$ to the pair of its initial and terminal points $(\alpha(0),\alpha(m))$~\cite{KI}.

\begin{defn}(\cite{KI})
The digital topological complexity number $\digtc_{\kappa}(X)$ is the least integer $l$ such that $U_{0}, U_{1}, ..., U_{l}$ is a cover of $X \times X$ and for any $1 \leq i \leq l$, admits digital continuous map $s_{i} : U_{i} \to X^{I_m}$ such that $\pi \circ s_{i} = id_{U_{i}}$.
\end{defn}

The digital continuity of $s_{i}$ is needed for the definition of the digital topological complexity. In order to ensure an adjacency relation between two digital paths is given: Let $\alpha_{1} : I_{m_1} \to X$ and $\alpha_{2} : I_{m_2} \to X$ be any two digital continuous paths in $X$. Then $\alpha_{1}$ and $\alpha_{2}$ are digital $\lambda-$connected on $X^{I_{m_1 + m_2}}$, if for all $t$ times, $\alpha_{1}(t)$ and $\alpha_{2}(t)$ are digital $\lambda-$connected. Here, $t$ can be considered different times in $\alpha_{1}$ and $\alpha_{2}$.

\begin{theo}(\cite{KI}) Digital topological complexity depends on only digital homotopy.
\end{theo}

\begin{theo}(\cite{KI}) If $(X , \kappa)$ is a digital path-connected space, then
\begin{displaymath}
\dcat_{\kappa} (X) \leq \digtc_{\kappa}(X) \leq \dcat_{\lambda} (X \times X),
\end{displaymath}
where $\lambda$ is an adjacency relation for the cartesian product space $X \times X$.
\end{theo}

This inequality implies that $\digtc_{\kappa}(X) \leq 2 (\dcat_{\kappa} (X))$.

\begin{defn}(\cite{BV})
Digital LS-category $\dcat_{\kappa}(X)$  of a space $(X, \kappa)$ is the least integer $l$ such that there exists a cover of $X$ by $l+1$ subsets $U_0 , U_1 , \ldots, U_l \subset X$ where each inclusion map $i_i: U_i \hookrightarrow X$ for $i = 0, \ldots , l$ is digital $\kappa$-nullhomotopic. 
\end{defn}

\begin{theo}(\cite{BV}) Digital LS-category depends on only digital homotopy as well.
\end{theo}

Throughout this work, we consistently assume that a subset of $\mathbb{Z}^n$ has the possible maximal adjacency to preserve the adjacency relation on the product of spaces. In short, we use the notation $\dcat(X) = \dcat_{\kappa}(X)$ and $\digtc(X) = \digtc_{\kappa}(X)$ and express all the digital terms without indexing adjacency. 

\section{Main Results}
The projective product space has been introduced by Davis in~\cite{Davis} as the quotient space $P_{\bar{n}} = S_{\bar{n}} / (\bar{x} \sim -\bar{x}) = (S^{n_1} \times \cdots \times S^{n_r}) /  ((x_1 , \ldots , x_r) \sim (-x_1 , \ldots , -x_r))$ concerning the diagonal action of $\mathbb{Z}_2$ on $S_{\bar{n}}$ where positive integers $n_1 \leq \ldots \leq n_r$. In the case of r = 1, as known, the space $P_{\bar{n}}$ equals to  usual real projective space $P^{n_1}$.

In order to define the digital projective product space, we use the concept of digital spheres. A digital $0$-dimensional space is a disconnected digital space $S^0 (x,y)$ with just two points $x$ and $y$ and the join $S_{\text{min}} ^n = S_0 ^0 \oplus S_1 ^0 \oplus \ldots \oplus S_n ^0$ of $n$-copies of zero dimensional surfaces $S^0$ is called a minimal digital $n$-sphere in~\cite{Evako}. As a result of this definition, we set the following notations:

\begin{itemize}
\item $S_{\text{min}} ^n = \left\{  x_i = (x_{i_0} , \ldots , x_{i_n}) \in \mathbb{Z}^{n+1} \right\}$ where $|x_{i_j}| = 1$ if $i = j$ and $0$ otherwise,
\end{itemize}

\begin{itemize}
\item $S_k ^0 = \{ x = (0, \ldots , 0 , x_k , 0 , \ldots , 0) \subset \mathbb{Z}^{n+1} \text{ : } \Vert x \Vert = 1 \}$ for $k \leq n$.
\end{itemize}
Notice that a $n$-dimensional digital sphere in $\mathbb{Z}^{n+1}$ has $2n + 2$ vertices.

\begin{ex} Digital spheres $S_{\text{min}} ^1$ and $S_{\text{min}} ^2$ that are modified from figures in ~\cite{Evako, LOS} are illustrated: 
\begin{flalign*}
S_{\text{min}} ^1 &= \{(1,0), (-1,0), (0,1), (0,-1)\} \\
S_{\text{min}} ^2 &= \{(1,0,0), (-1,0,0), (0,1,0), (0,-1,0), (0,0,1), (0,0,-1)\}
\end{flalign*}

\begin{figure}[ht]
\begin{minipage}[b]{0.27\linewidth}
\centering
\includegraphics[width=\textwidth]{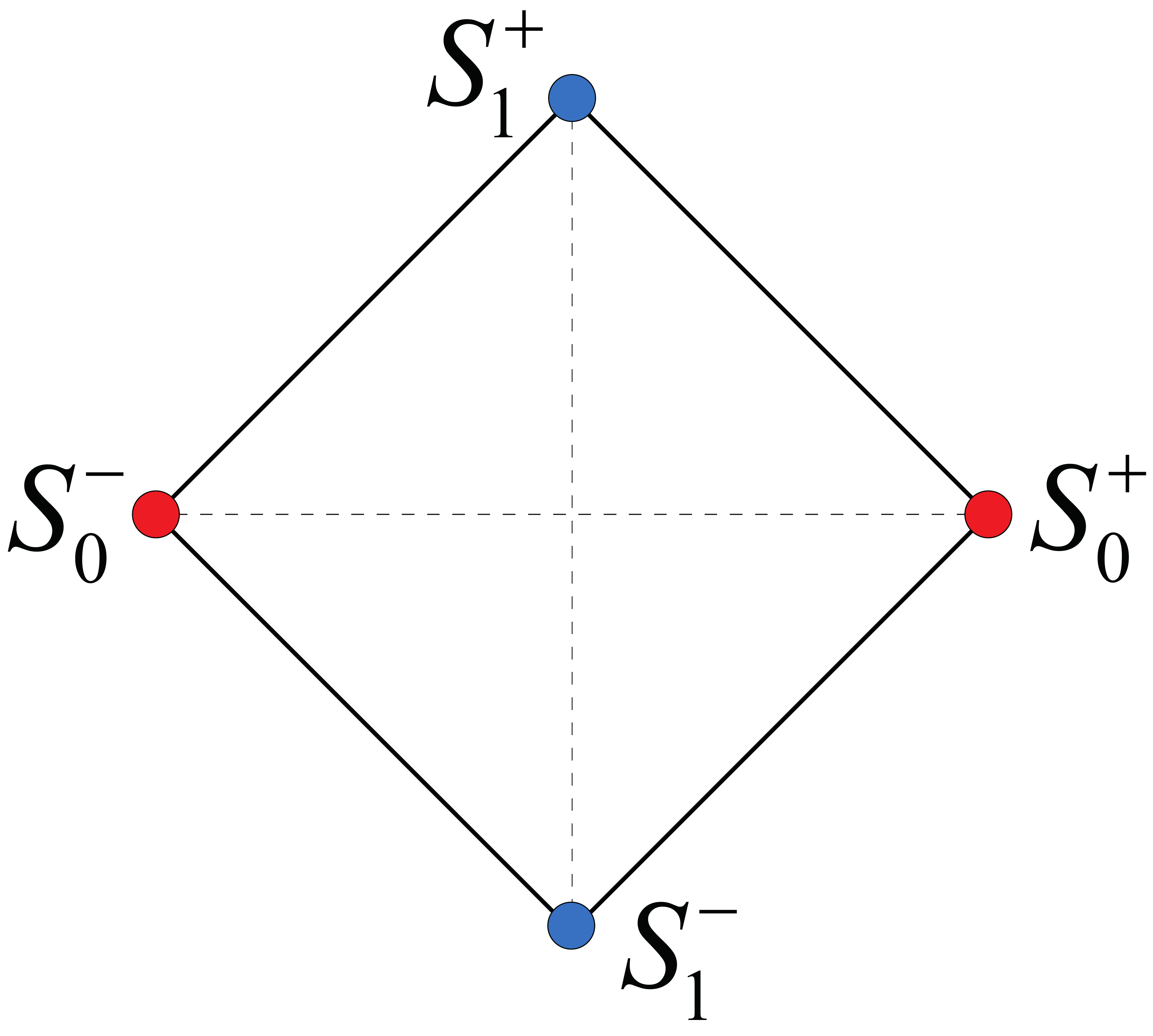}
\caption{$S_{\text{min}} ^1$}
\label{fig:figure1}
\end{minipage}
\hspace{3cm}
\begin{minipage}[b]{0.27\linewidth}
\centering
\includegraphics[width=\textwidth]{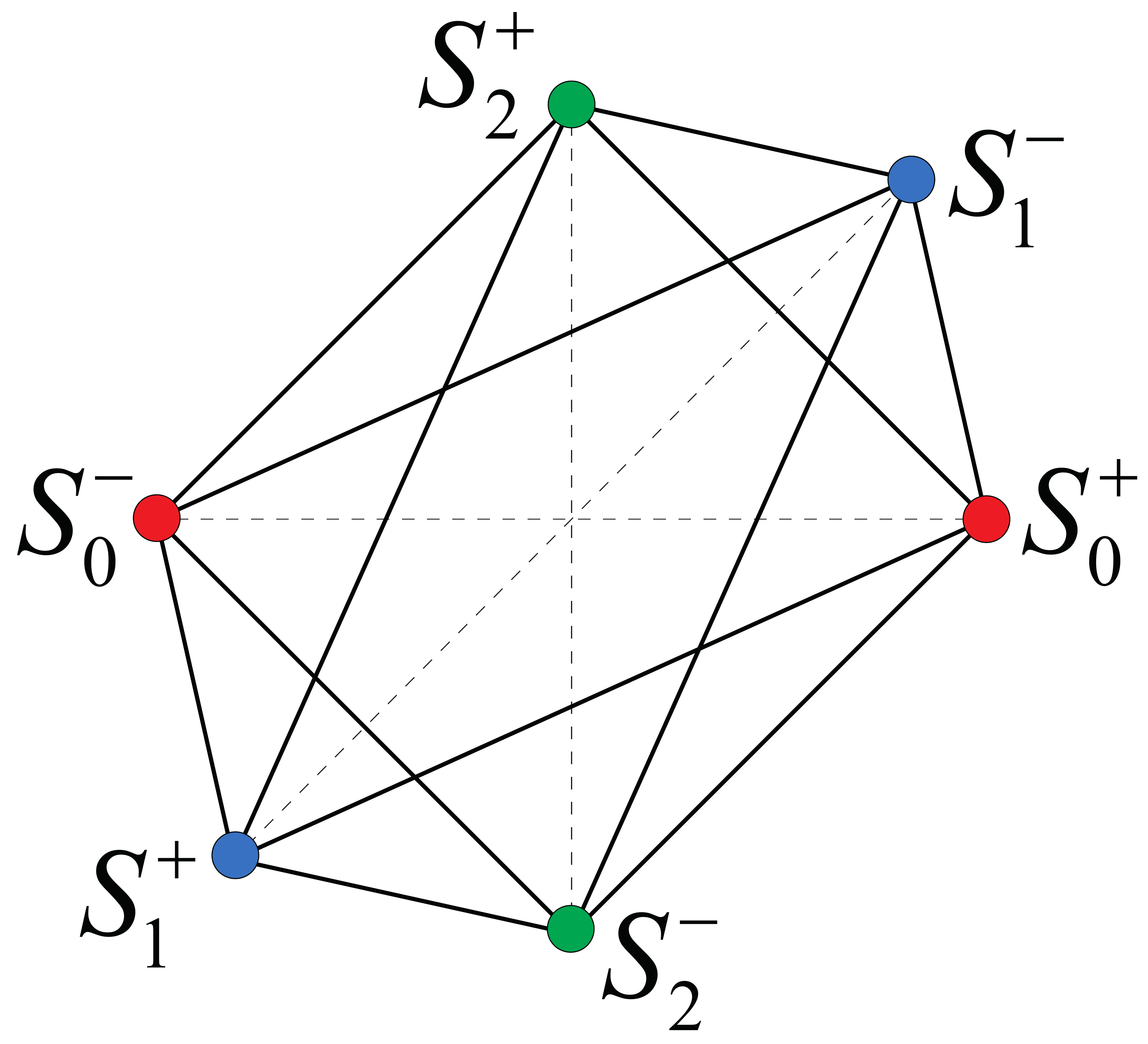}
\caption{$S_{\text{min}} ^2$}
\label{fig:figure2}
\end{minipage}
\end{figure}
\end{ex}

We present the quotient space,
\begin{displaymath}
\mathbb{Z}P^n := \dfrac{S_{\text{min}} ^n}{x \sim -x}
\end{displaymath}
with respect to the diagonal action of $\mathbb{Z}_2$ on $S_{\text{min}} ^n$ as the digital projective space. We denote it by $\text{d-}P^n$.

Projective product space in digital topology sense, we introduce the diagonal action of $\mathbb{Z}_2$ on $\digs$ as the digital projective product space where $\digs := S_{\text{min}} ^{n_1} \times \cdots \times S_{\text{min}} ^{n_r}$ called the product of digital spheres and $\bar{x} = (x_1 , \ldots , x_r) \in \digs$. We signify digital projective product space by
\begin{displaymath}
\digp = \dfrac{\digs} {\bar{x} \sim -\bar{x}} = \dfrac{S_{\text{min}} ^{n_1} \times \cdots \times S_{\text{min}} ^{n_r} } { (x_1 , \ldots , x_r) \sim (-x_1 , \ldots , -x_r) }.
\end{displaymath}
The dimension $\dim( \digp) = \dim (\digs) = \sum n_i$ and if r = 1, then the space $\digp$ coincides with the digital projective space $\text{d-}P^{n_1}$ which we generalized.

\begin{theo} Let $\digp := (S_{\text{min}} ^{n_1} \times \cdots \times S_{\text{min}} ^{n_r}) / (\bar{x} \sim -\bar{x})$ be digital projective product space where $\bar{n} = (n_1 , \ldots , n_r)$ and $n_1 \leq \ldots \leq n_r$. Then the digital Lusternik-Schnirelman category of $\digp$ satisfies that $\dcat (\digp) \leq n_1 + r - 1$.
\end{theo}

\begin{proof} Let $\bar{n} = (n_1 , \ldots , n_r)$, where $n_1 \leq \ldots \leq n_r$.
\begin{center}
$|\bar{n}| := \dim(\digp) = \dim(\digs) = \sum{n_i}$ and $l(\bar{n}) = r$.
\end{center}
For $k \leq n$, we consider the sets $S_k ^0$ and $S_{\text{min}} ^n$ and fix the following general notation:
\begin{itemize}
\item $S_k ^+ = \{ x = (0, \ldots , 0 , x_k , 0 , \ldots , 0) \subset \mathbb{Z}^{n+1} \text{ : } x_k > 0 \}$,
\end{itemize}

\begin{itemize}
\item $S_k ^- = \{ x = (0, \ldots , 0 , x_k , 0 , \ldots , 0) \subset \mathbb{Z}^{n+1} \text{ : } x_k < 0 \}$,
\end{itemize}

\begin{itemize}
\item $S_k ^0 = S_k ^+ \cup S_k ^-$,
\end{itemize}

\begin{itemize}
\item $A_k ^+ = (0, \ldots , 0 , \underbrace{1}_{\text{$k$. entry}} , 0 , \ldots , 0)$ (North pole) and
\end{itemize}

\begin{itemize}
\item $A_k ^- = (0, \ldots , 0 , \underbrace{-1}_{\text{$k$. entry}} , 0 , \ldots , 0)$ (South pole),
\end{itemize}

Notice that
\begin{flalign*}
S_{\text{min}} ^n &= S_0 ^0 \oplus S_1 ^0 \oplus \ldots \oplus S_n ^0 = S^0 (S_0 ^+ , S_0 ^-) \oplus S_{\text{min}} ^{n-1} = S_0 ^+ \cup S_{\text{min}} ^{n-1} \cup S_0 ^- \\
&= S_n ^+ \cup S_{\text{min}} ^{n-1} \cup S_n ^- = S_n ^+ \cup S_{n-1} ^+ \cup \ldots \cup S_0 ^+ \cup S_0 ^- \cup \ldots \cup \cup S_{n-1} ^- \cup S_n ^-
\end{flalign*}
and define the map $\epsilon_k : S_k ^0 \to \pm 1$ such that
\begin{center}
$\epsilon_k (0 , \ldots , x_k , \ldots , 0) = 
\begin{cases}
\text{ }\text{ }1,  &\text{ if } x_k > 0 \\
-1, &\text{ if } x_k < 0.
\end{cases}$
\end{center}
Note that $\epsilon_k (-x) = -\epsilon_k (x)$ for any $x \in S_k ^0$.

Let $\rho(A,B) : I_{m_1} \to S_{\text{min}} ^n$ be the digital path from $A$ to $B$ for except antipodal points $A,B$ $(A \neq -B)$ of digital sphere $S_{\text{min}} ^n$. Say $\rho(-A,-B) = -\rho(A,B)$. Let also fixed meridian $\sigma_0 : I_{m_2} \to S_{\text{min}} ^k$ digital path from $A_k ^-$ (south pole) to $A_k ^+$ (north pole) such that $\sigma_0 (0) = A_k ^-$ and $\sigma_0 (m_2) = A_k ^+$.

We will define a cover $\bigcup_{i = 0} ^{n_1 + r - 1} U_i$ of $\digs := S_{\text{min}} ^{n_1} \times \cdots \times S_{\text{min}} ^{n_r}$ where $n_1 \leq \ldots \leq n_r$.

\begin{itemize}
	\item For $0 \leq i \leq n_1 - 1$, we set $U_i = S_i ^0 \times \prod_{q=2} ^r S_{\text{min}} ^{n_q - 1}$.
\end{itemize}

\begin{itemize}	
	\item For a subset $J \subset \{ 1, \ldots , r \}$, we denote $|J|$ the cardinality of $J$ and consider
\begin{center}
$Q_J = \{ \bar{u} \in \digs \text{ : } u_j \in S_{n_j} ^0 \text{ if } j \in J, \text{ }u_j \in S_{\text{min}} ^{n_j} \text{ if } j \notin J \}$
\end{center}
Realize that for $J \neq J'$ with indifferent cardinality $|J| = |J'|$, the sets $Q_J$ and $Q_{J'}$ are disjoint. By the way, we use the inspired notation by~\cite{CP}.

\noindent For $i = n_1 - 1 +k$, where $k = 1, \ldots , r-1$, we set $U_i = U_{n_1 - 1 +k} = \bigcup_{|J|=k} Q_J$.
\end{itemize}

\begin{itemize}
\item For $i = n_1 + r - 1$, we set $U_i = U_{n_1 + r - 1} = S_{n_1} ^0 \times \cdots \times S_{n_r} ^0 = S_{\bar{n}} ^0$. We have a cover of $\digs = U_0 \cup \ldots \cup U_{n_1 + r - 1}$ by subsets. We now define, $0 \leq i \leq n_1 + r - 1$, a digital homotopy function by
\begin{center}
$h_i :U_i \to (\digp)^{I_{m_1 + m_2}}$.
\end{center}
The class of an element $\bar{u} \in \digs$ in $\digp$ is denoted by $[\bar{u}]$. 
\end{itemize}

\underline{For $0 \leq i \leq n_1 - 1$:} For $\bar{u}=(u_1,\dots,u_r) \in S_i ^0 \times \prod_{q=2} ^r S_{\text{min}} ^{n_q - 1}$ and $t \in I_{m_1 + m_2}$, we set 
\begin{displaymath}
h_i(\bar{u},t) = \left[ \rho(u_1 , A^i _{\epsilon_i(u_1)})(t) , \rho(u_2 , A^{n_2} _{\epsilon_i(u_1)})(t) , \ldots , \rho(u_r , A^{n_r} _{\epsilon_i(u_1)})(t) \right].\\
\end{displaymath}

\underline{For $n_1 \leq i \leq n_1 + r - 2$:} Recall that we write $i=n_1-1+k$ with $k = 1, \ldots , r - 1$ and that $U_i = U_{n_1 - 1 + k} = \bigcup_{|J|=k} Q_J$. We give a digital homotopy function by
\begin{displaymath}
h_J: Q_J  \to (\digp)^{I_{m_1 + m_2}}. 
\end{displaymath}
Set $j_0 = \text{min }J$. If $\bar{u} \in Q_J$, then $u_{j_0} \in S_{n_{j_0}} ^0$.
\noindent
For $\bar{u} \in Q_J$ and $t \in I_{m_1 + m_2}$, $0 \leq t \leq m_1$, we set
\begin{displaymath}
h_J(\bar{u},t) = \left[ \rho(u_1 , B(u_1))(t) , \ldots , \rho(u_q , B(u_q))(t) , \ldots , \rho(u_r , B(u_r))(t) \right],
\end{displaymath}
where, for $1 \leq q \leq r$, $B(u_q) = \begin{cases} A^{n_q} _{\epsilon(u_q)}, & \text{ if } q \in J, \\ A^{n_q} _{\epsilon(u_{j_0})}, & \text{ if } q \notin J. \end{cases}$ \\

We distinguish two cases for $t \in I_{m_1 + m_2}$, $m_1 \leq t \leq m_1 + m_2$.\\
If $\epsilon(u_{j_0})  = 1$, then $h_J(\bar{u},t) = [\omega_1 , \ldots , \omega_r]$ where, for $1 \leq q \leq r$, 
\begin{displaymath}
\omega_q = \begin{cases} \sigma_0 (t - m_1), & \text{ if } \epsilon(u_{q}) = -1 \text{ and } q \in J \\ A^{n_q} _{\epsilon({u_{j_0}})}, & \text{otherwise}. \end{cases}
\end{displaymath}
If $\epsilon(u_{j_0})  = -1$, then
$h_J(\bar{u},t) = [\omega_1 , \ldots , \omega_r]$ where, for $1 \leq q \leq r$, 
\begin{displaymath}
\omega_q = \begin{cases} \sigma_0 (t - m_1), & \text{ if } \epsilon(u_{q}) = 1 \text{ and } q \in J \\ A^{n_q} _{-\epsilon({u_{j_0}})}, & \text{otherwise}. \end{cases}
\end{displaymath}

We obtain $\left[ A^{n_1} _{\epsilon(u_1)} , \ldots , A^{n_j} _{\epsilon(u_{j})} , \ldots , A^{n_r} _{\epsilon(u_r)} \right] = \left[ A^{n_1} _{-\epsilon(u_1)} , \ldots , A^{n_j} _{-\epsilon(u_{j})} , \ldots , A^{n_r} _{-\epsilon(u_r)} \right]$ for $t = m_1$. This provides us a well-defined digital continuous map on $Q_J \times I_{m_1 + m_2}$.

We define $h_i$ on $U_i = U_{n_1 - 1 + k} = \bigcup_{|J| = k} Q_J$ by setting $h_i |_{Q_J \times I_{m_1 + m_2}} = h_J$, for $k = n_1 - 1 + k$.

\underline{For $i = n_1 + r - 1$:} For $\bar{u} \in S_{n_1} ^0 \times \cdots \times S_{n_r} ^0 = S_{\bar{n}} ^0$ and $t \in I_{m_1 + m_2}$, we set
\begin{displaymath}
h_i(\bar{u},t) = \left[ \rho(u_1 , A^{n_1} _{\epsilon(u_1)})(t) , \ldots , \rho(u_r , A^{n_r} _{\epsilon(u_r)})(t) \right].
\end{displaymath}
We seperate two cases for $t \in I_{m_1 + m_2}$ and $m_1 \leq t \leq m_1 + m_2$.\\
If $\epsilon(u_{1})  = 1$, then $h_i(\bar{u},t) = [A^{n_1} _{\epsilon(u_1)}, \omega_2,\ldots , \omega_r]$ where, for $2\leq q \leq r$, 
\begin{displaymath}
\omega_q =  \begin{cases} A^{n_q} _{\epsilon(u_q)}, \text{ if } \epsilon(u_q) = 1 \\ \sigma_0 (t - m_1), \text{ if } \epsilon(u_q) = -1 \end{cases}
\end{displaymath}
If $\epsilon(u_{1})  = -1$, then $h_i(\bar{u},t) = [A^{n_1} _{-\epsilon(u_1)}, \omega_2,\ldots , \omega_r]$ where, for $2\leq q \leq r$, 
\begin{displaymath}
\omega_q =  \begin{cases} A^{n_q} _{-\epsilon(u_q)}, \text{ if } \epsilon(u_q) = -1 \\ \sigma_0 (t - m_1),  \text{ if } \epsilon(u_q) = 1 \end{cases}
\end{displaymath}

For $t = m_1$, $\left[ A^{n_1} _{\epsilon(u_1)} , \ldots, A^{n_r} _{\epsilon(u_r)} \right] = \left[ A^{n_1} _{-\epsilon(u_1)}, \ldots, A^{n_r} _{-\epsilon(u_r)} \right]$. This gives a well-defined the digital continuous map on $S_{\bar{n}} ^0 \times I_{m_1 + m_2}$.

We have $h_i (\bar{u},t) = h_i (-\bar{u},t)$ for any $t \in I_{m_1 + m_2}$ and for any $i$ where $\bar{u} \in U_i$ and and $\bar{u}\in U_i$. According to the maps, we obtain $h_i (\bar{u},0) = [\bar{u}]$ for $0 \leq i \leq n_1+r-1$,
\begin{displaymath}
h_i (\bar{u},m_1 + m_2) =
\begin{cases}
[A_+ ^{i}, A_+ ^{n_2},  \ldots, A_+ ^{n_r}], &0 \leq i \leq n_1 - 1 \\
[A_+ ^{n_1}, A_+ ^{n_2},  \ldots, A_+ ^{n_r}], &n_1 \leq i \leq n_1 + r - 1.
\end{cases}
\end{displaymath}
For any $i$ and $\bar{u}\in U_i$, $h_{i_{\bar{u}}} : I_{m_1 + m_2} \to \digp$ defined by $h_{i_{\bar{u}}} (t) = h_i (\bar{u} , t)$ is digitally continuous and for any $t \in I_{m_1 + m_2}$, $h_{i_t} : U_i \to \digp$ defined by $h_{i_t}(\bar{u}) = h_i (\bar{u} , t)$ is digitally continuous as well. Hence, for any $i$, $V_i={U_i}/{\sim}$ is a subset of $\digp$ and we gain a digital homotopy function $\bar{h}_i : V_i \to (\digp)^{I_{m_1 + m_2}}$ such that $\bar{h}_i ([\bar{u}],0) = [\bar{u}]$ for $0 \leq i \leq n_1+r-1$ and
\begin{displaymath}
\bar{h}_i ([\bar{u}],m_1 + m_2) =
\begin{cases}
[A_+ ^{i}, A_+ ^{n_2},  \ldots, A_+ ^{n_r}], &0 \leq i \leq n_1 - 1 \\ 
[A_+ ^{n_1}, A_+ ^{n_2},  \ldots, A_+ ^{n_r}], &n_1 \leq i \leq n_1 + r - 1.
\end{cases}
\end{displaymath}
For any $i$ and $[\bar{u}] \in V_i$, $\bar{h}_{i_{[\bar{u}]}} : I_{m_1 + m_2} \to \digp$ defined by $\bar{h}_{i_{[\bar{u}]}} (t) = \bar{h}_i ([\bar{u}] , t)$ is digitally continuous and for any $t \in I_{m_1 + m_2}$, $\bar{h}_{i_t} : V_i \to \digp$ defined by $\bar{h}_{i_t}([\bar{u}]) = \bar{h}_i ([\bar{u}] , t)$ is also digitally continuous. In addition, we obtain a cover of  $\bigcup_{i=0} ^{n_1 + r - 1} V_i$ of $\digp$ and each inclusion map $V_i \hookrightarrow \digp$ is digitally nullhomotopic. Therefore, we prove that $\dcat (\digp) \leq n_1 + r - 1$.
\end{proof}

\begin{ex} We specify the construction above for $n_1=1$, $n_2=2$ and $r=2$. In other words, we show that 
\begin{displaymath}
\dcat(\text{d-}P) = \dcat \left( \dfrac{S_{\text{min}} ^{1} \times S_{\text{min}} ^{2}}{\sim} \right) \leq n_1+r-1 = 2. 
\end{displaymath}

\begin{figure}[H]
\centering
\scalebox{0.08}{\includegraphics{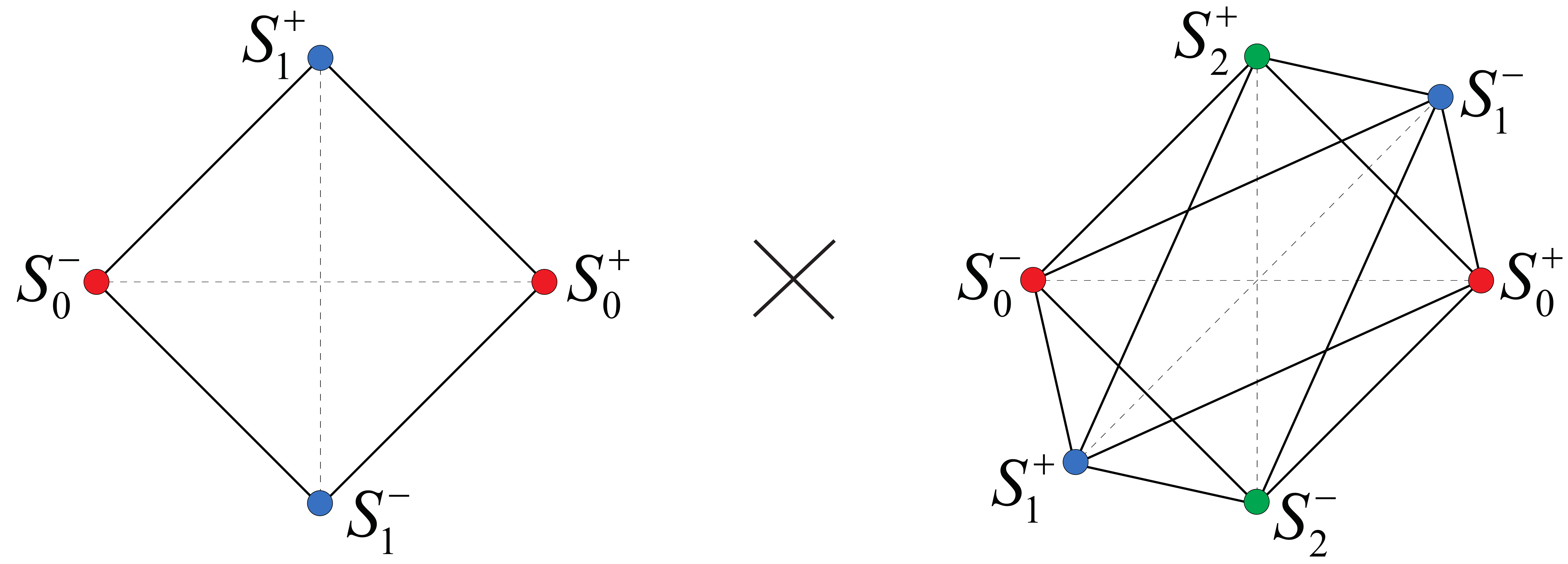}}
\caption{$S_{\text{min}} ^{1} \times S_{\text{min}} ^{2}$}
\label{sp2}
\end{figure}

As before, assume that $\rho (A,B):I_{m_1} \to S_{\text{min}} ^{n}$ is the digital path from $A$ to $B$ for non-antipodal points $A,B \in S_{\text{min}} ^{n}$ with $A \neq B$. Note that $\rho (-A,-B) = -\rho (A,B)$. Moreover, let $\sigma_0 : I_{m_2} \to S_{\text{min}} ^{n}$ be the fixed meridian digital path from $A^n _-$ to $A^n _+$ with $\sigma_0 (0) = A^n _-$ and $\sigma_0 (1) = A^n _+$.

Fix the set $U_i$ in the following notation:
\begin{displaymath}
U_i = \begin{cases} S_0^{0} \times S_{\text{min}} ^{1}, &  i = 0 \\
S_1^{0} \times S_{\text{min}} ^{1} \cup S_{\text{min}} ^{0} \times S_2^{0}, & i = 1 \\
S_1^{0} \times S_2^{0}, & i = 2
\end{cases}
\end{displaymath}
and here we have $S_{\text{min}} ^{1} \times S_{\text{min}} ^{2} \subset \bigcup_{i=0} ^{2} U_i$.

For $i = 0$, the digital homotopy function $h_0 : U_0 \times I_{m_1 + m_2} \to d$-$P = (S_{\text{min}} ^{1} \times S_{\text{min}} ^{2})/\sim$ is defined by
\begin{flalign*}
(u_1,u_2,t) \in S_0^{+} \times S_{\text{min}} ^{1} \times I_{m_1 +m_2}, \text{ }& h_0 (u_1,u_2,t) = \left[ \rho(u_1, A^0 _+)(t) , \rho(u_2, A^{2} _+)(t) \right],\\ \text{and} \text{ } 
(u_1,u_2,t) \in S_0^{-} \times S_{\text{min}} ^{1} \times I_{m_1 +m_2}, \text{ }& h_0 (u_1,u_2,t) = \left[ \rho(u_1, A^0 _-)(t) , \rho(u_2, A^{2} _-)(t) \right].
\end{flalign*}

For $i = 1$, the digital homotopy function $h_1 : U_1 \times I_{m_1 + m_2} \to d$-$P = (S_{\text{min}} ^{1} \times S_{\text{min}} ^{2})/\sim$ is given by
\begin{flalign*}
(u_1,u_2,t) \in S_1^{+} \times S_{\text{min}} ^{1} \times I_{m_1 + m_2}, \text{ }& h_{1} (u_1,u_2,t) = \left[ \rho(u_1, A^{1} _+)(t) , \rho(u_2, A^{2} _+)(t) \right],\\
(u_1,u_2,t) \in S_1^{-} \times S_{\text{min}} ^{1} \times I_{m_1 + m_2}, \text{ }& h_{1} (u_1,u_2,t) = \left[ \rho(u_1, A^{1} _-)(t) , \rho(u_2, A^{2} _-)(t) \right],\\
(u_1,u_2,t) \in S_{\text{min}} ^{0} \times S_2^{+} \times I_{m_1 + m_2}, \text{ }& h_{1} (u_1,u_2,t) = \left[ \rho(u_1, A^{1} _+)(t) , \rho(u_2, A^{2} _+)(t) \right]\\ \text{and} \text{ }
(u_1,u_2,t) \in S_{\text{min}} ^{0} \times S_2^{-} \times I_{m_1 + m_2}, \text{ }& h_{1} (u_1,u_2,t) = \left[ \rho(u_1, A^{1} _-)(t) , \rho(u_2, A^{2} _-)(t) \right].
\end{flalign*}

For $i = 2$, the digital homotopy function $h_{2} : U_{2} \times I_{m_1 + m_2} \to d$-$P = (S_{\text{min}} ^{1} \times S_{\text{min}} ^{2})/\sim$ is considered as follows:
\begin{flalign*}
(u_1,u_2,t) \in S_1^{+} \times S_2^{+} \times I_{m_1 + m_2} ,\text{ }& h_{2}(u_1,u_2,t) = \left[ \rho(u_1, A^{1} _+)(t) , \rho(u_2, A^{2} _+)(t) \right]\\ \text{and} \text{ }
(u_1,u_2,t) \in S_1^{-} \times S_2^{-} \times I_{m_1 + m_2} ,\text{ }& h_{2}(u_1,u_2,t) = \left[ \rho(u_1, A^{1} _-)(t) , \rho(u_2, A^{2} _-)(t) \right].
\end{flalign*}
	
If $(u_1,u_2,t) \in S_1^{+} \times S_2^{-} \times I_{m_1 + m_2}$, then we have
$$h_{2} (u_1,u_2,t) = \begin{cases} \left[ \rho(u_1, A^{1} _+)(t) , \rho(u_2, A^{2} _-)(t) \right], & 0 \leq t \leq m_1 \\
\left[ A^{1} _+ , \sigma_0 (t-m_1) \right], & m_1 \leq t \leq m_1+m_2.
\end{cases}$$
Notice that $h_{2}$ is well-defined digital continuous on $S_1^{+} \times S_2^{-} \times I_{m_1 + m_2}$ since there exists $\left[ A^{1} _+ , A^{2} _- \right] = \left[ A^{1} _+ , A^{2} _- \right]$ at $t = m_1$.\\
If $(u_1,u_2,t) \in S_1^{-} \times S_2^{+} \times I_{m_1 + m_2}$, then we get
$$h_{2} (u_1,u_2,t) = \begin{cases} \left[ \rho(u_1, A^{1} _-)(t) , \rho(u_2, A^{2} _+)(t) \right], & 0 \leq t \leq m_1 \\
\left[ A^{1} _+ , \sigma_0 (t-m_1) \right], & m_1 \leq t \leq m_1+m_2.
\end{cases}$$ 
In this case, $h_{2}$ is well-defined digital continuous on $S_1^{-} \times S_2^{+} \times I_{m_1 + m_2}$ since we have $\left[ A^{1} _- , A^{2} _+ \right] = \left[ A^{1} _+ , A^{2} _- \right]$ at $t = m_1$.
 
Given any $(u_1, u_2) \in U_i$, then we have $(-u_1, -u_2) \in U_i$ and $h_i (u_1, u_2,t) = h_i (-u_1,-u_2,t)$ for any $t \in I_{m_1 + m_2}$ and any $0 \leq i \leq 2$. We obtain $h_i (u_1, u_2,0) = [u_1, u_2]$ for $0 \leq i \leq 2$, and
\begin{displaymath}
h_i(u_1, u_2,m_1 + m_2) =
\begin{cases}
\left[ A^{0} _+ , A^{2} _+ \right], &i = 0, \\
\left[ A^{1} _+ , A^{2} _+ \right], &i=1,2.
\end{cases} 
\end{displaymath}

The map $h_{i_{(u_1, u_2)}} : I_{m_1 + m_2} \to \text{d-}P$ is defined by $h_{i_{(u_1, u_2)}}(t) = h_i((u_1, u_2), t)$ is digital continuous for any $i$ and $(u_1, u_2) \in U_i$, and the map $h_{i_t}:U_i \to \text{d-}P$ given by $h_{i_t}((u_1, u_2)) = h_i((u_1, u_2), t)$ is digital continuous for any $t \in I_{m_1 + m_2}$. These yield a digital homotopy function $\bar{h}_i : V_i \times I_{m_1 + m_2} \to \text{d-}P$ such that  $\bar{h}_i ([u_1, u_2],0) = [u_1, u_2]$ for $0 \leq i \leq 2$, and
\begin{displaymath}
\bar{h}_i ([u_1, u_2],m_1 + m_2) =
\begin{cases}
\left[ A^{0} _+ , A^{2} _+ \right], &i= 0, \\
\left[ A^{1} _+ , A^{2} _+ \right], &i = 1,2.
\end{cases}
\end{displaymath}
Thus, we obtain $\bigcup_{i = 0} ^{n_1 +1} V_i$ where $V_i= U_i / \sim$ is a cover of $\text{d-}P$ provides that each inclusion $V_i \hookrightarrow \text{d-}P$ is nullhomotopic. Namely, each $V_i \subset \text{d-}P$ is a categorical subset. Hence, we conclude that $\dcat (\text{d-}P) \leq 2$.
\end{ex}

\begin{cor} Given $\digp$ digital projective product space, we have
\begin{displaymath}
\digtc(\digp) \leq \dcat (\digp) \leq 2(n_1 + r - 1).
\end{displaymath}
\end{cor}

\begin{defn}
A digitally continuous map $f: \mathbb{Z}^n \times \mathbb{Z}^n \to \mathbb{Z}^k$ is called a digital non-singular map, if it holds the following conditions:
\begin{itemize}
	\item $f(ax,by) =ab f(x,y)$ for every $x,y \in \mathbb{Z}^n$ and $a,b \in \mathbb{Z}$.
\end{itemize}	
\begin{itemize}	
	\item $f(x,y) = 0$ implies that either $x = 0$ or $y = 0$.
\end{itemize}
\end{defn}

\begin{prop} If one can find a digital non-singular map $f: \mathbb{Z}^{n+1} \times \mathbb{Z}^{n+1} \to \mathbb{Z}^{k+1}$ where $n+1 \leq k$, then $\text{d-}P^n$ has a motion planner with $k$ local acts, which means 
 \begin{displaymath}
 \dcat (\text{d-}P^n) \leq k.
 \end{displaymath}
\end{prop}

\begin{proof} We assume that $\theta: \mathbb{Z}^{n+1} \times \mathbb{Z}^{n+1} \to \mathbb{Z}$ be a scalar digital continuous map with the property that $\theta(au, bv) = ab\theta(u, v)$ for all $(u, v) \in S_{\text{min}} ^{n} \times S_{\text{min}} ^{n}$ and $a, b \in \mathbb{Z}$.

Let $U_{\theta} \subset \text{d-}P^n \times \text{d-}P^n$ represents the set of all pairs $(u, v)$ of points in $\mathbb{Z}^{n+1}$ such that $u \neq v$ and $\theta(u, v) \neq 0$ for some points $u, v \in S_{\text{min}} ^{n}$.

We assert that there is continuous motion planning in $U_{\theta}$. Namely, there exists  digital continuous map $s$ defined on $U_{\theta}$ with values for the space of digital continuous paths in the digital projective space $\text{d-}P^n$ such that for each pair $(u, v) \in U_{\theta}$ the digital path $s(u, v)(t)$, $t \in I_m$, begins at point $u$ and terminates at point $v$. We may provide points in $S_{\text{min}} ^{n}$ such that $\theta(u, v) > 0$ by considering the construction of $\text{d-}P^n$.  In this instance, we may take $-u$, $-v$ instead of $u$, $v$. Notice that $u$, $v$ and equivalently $-u$, $-v$ dictate the indifferent orientation of the plane based on these points. The intended motion planning digital map $s$ occurs in rotating $u$ to $v$ in this plane, in the positively directed by orientation. 

Furthermore, we suppose that $\theta: \mathbb{Z}^{n+1} \times \mathbb{Z}^{n+1} \to \mathbb{Z}$ is called positive if we have $\theta(u, u) > 0$ for any $u \in \mathbb{Z}^{n+1}$. We may choose a lightly larger set $U_{\theta}^{\prime} \subset \text{d-}P^n \times \text{d-}P^n$ derive from $U_{\theta}$ which is identified as the set of all pairs of points of the form $(u, v)$ where $\theta(u,v) \neq 0$ for some $u, v \in S_{\text{min}} ^{n}$. We see that the set $U_{\theta}^{\prime}$ contains all the pairs of the points $(u,u)$. We describe the digital path from $u$ to $v$ for $u \neq v$ as rotating from $u$ to $v$ in the plane, based on $u$ and $v$ in the positively indicated by the orientation. We take the constant digital path at point $u$. Therefore, we preserve digital continuity. 

A digital non-singular map $f: \mathbb{Z}^{n+1} \times \mathbb{Z}^{n+1} \to \mathbb{Z}^{k}$ admits $k$ scalar digital map $\theta_1, \ldots, \theta_k : \mathbb{Z}^{n+1} \times \mathbb{Z}^{n+1} \to \mathbb{Z}$ and the specified as before $U_{\theta_i}$ cover the product $\text{d-}P^n \times \text{d-}P^n$ except the diagonal. Due to $n+1 < k$, we may use such an $f$ as the initial digital non-singular map with the property that for any $u \in \mathbb{Z}^{n+1}$ the first coordinate $\theta_1(u, u)$ is positive. The sets $U_{\theta_1}^{\prime}, U_{\theta_2}, \ldots, U_{\theta_k}$ is a cover of $\text{d-}P^n \times \text{d-}P^n$. We have stated an explicit motion planning instructions over any of these sets. Hence, we get the inequality $\digtc(\text{d-}P^n) \leq k$.

\end{proof}

\begin{theo} If $\digp := S_{\text{min}} ^{n_1} \times \cdots \times S_{\text{min}} ^{n_r} / (\bar{x} \sim -\bar{x})$ is the digital projective product space where $\bar{n} = (n_1 , \ldots , n_r)$ and $n_1 \leq \ldots \leq n_r$, then we have
\begin{displaymath}
\digtc (\digp) \leq \digtc (\dign) + \sum_{q=2} ^r \digtc (S_{\text{min}} ^{n_q}).
\end{displaymath}
\end{theo}

\begin{proof} We consider the cartesian product of two digital projective product spaces $\digp \times \digp = (\digs \times \digs) / \sim$ as the quotient of $(S_{\text{min}} ^{n_1} \times S_{\text{min}} ^{n_1}) \times \cdots \times (S_{\text{min}} ^{n_r} \times S_{\text{min}} ^{n_r})$ by using the well-known isomorphism for the relation
\begin{equation}\label{relation}
(u_1,v_1,\dots,u_r,v_r)\sim (u'_1,v'_1,\dots,u'_r,v'_r) \Leftrightarrow \left\{\begin{array}{rcc} 
\forall i& u_i=u'_i \text{ and } v_i=v'_i \\
\text{or } \forall i& u_i=-u'_i \text{ and } v_i=v'_i \\
\text{or } \forall i& u_i=u'_i \text{ and } v_i=-v'_i \\
\text{or } \forall i& u_i=-u'_i \text{ and } v_i=-v'_i .\\
\end{array}\right.
\end{equation}
We set the construction of motion planners for the digital projective space $\dign$ and for a digital sphere $S_{\text{min}} ^{n_q}$ which has a vision from ~\cite{FTY} and ~\cite{Far}, respectively. We will get a motion planner on $\digp$ by gathering them.

Assume that $\digtc (\dign)=k$. According to Proposition 3.5, there exists a digital non-singular map $\theta=(\theta_2,\cdots,\theta_k) : \mathbb{Z}^{n_1 + 1} \times \mathbb{Z}^{n_1 + 1} \to \mathbb{Z}^{k+1}$. The $k+1$ scalar digital maps $\theta_2 , \theta_1 , \ldots , \theta_{k}: \mathbb{Z}^{n_1 + 1} \times \mathbb{Z}^{n_1 + 1} \to \mathbb{Z}$ have the property $\theta_i(au_1,bv_1)=ab\theta(u_1,v_1)$ for $(u_1,v_1) \in S_{\text{min}} ^{n_1} \times S_{\text{min}} ^{n_1}$ and $a,b \in \mathbb{Z}$ and do not become zero simultaneously. We suppose that $\theta_2(u_1,u_1)>0$ for any $u_1 \in S_{\text{min}} ^{n_1}$ from the definition of the digital sphere for $n_1 + 1 < k$. Let
\begin{flalign*}
U_0 &= \left\{ (u_1,v_1) \in S_{\text{min}} ^{n_1} \times S_{\text{min}} ^{n_1} \text{ : } \theta_2 (u_1,v_1) \neq 0 \right\} \hspace{1cm} ((u_1,u_1) \in U_0) \\ 
U_i &= \left\{ (u_1,v_1) \in S_{\text{min}} ^{n_1} \times S_{\text{min}} ^{n_1} \text{ : } \text{for all } 0 \leq n < i, \theta_n (u_1,v_1) = 0 \text{ and }\theta_i (u_1,v_1) \neq 0 \right\},
\end{flalign*}
where $1 \leq i \leq k-1$.

Notice that all the sets are compatible with the equivalence relation on $S_{\text{min}} ^{n_1} \times S_{\text{min}} ^{n_1}$ deduced by the antipodal relation $u_1 \sim -u_1$ on $S_{\text{min}} ^{n_1}$.

Note that all the sets $U_i$ are disjoint and that $\bigcup_{i=0} ^k U_i$ is a cover of $S_{\text{min}} ^{n_1} \times S_{\text{min}} ^{n_1}$.

Let $\rho(A,B): I_{m_1} \to S_{\text{min}} ^n$ be the digital path from $A$ to $B$ for except antipodal points $A,B$ $(A \neq -B)$ of digital sphere $S_{\text{min}} ^n$. Discern that $\rho(-A , -B) = -\rho(A,B)$ and that $\rho(A,A)$ is the constant digital path.

For $0 \leq i \leq k$, we define the map $\psi_i: U_i \to (\dign)^{I_{m_1 + m_2}}$ by
\begin{displaymath}
\psi_i (u_1,v_1) = \begin{cases} 
[\rho(u_1,v_1)], &\text{ if } \theta_i (u_1,v_1) > 0 \\
[\rho(-u_1,v_1)], &\text{ if } \theta_i (u_1,v_1) < 0.
\end{cases}
\end{displaymath}
We obtain $\theta_0 (u_1,u_1) > 0$ for any $u_1 \in S_{\text{min}} ^n$ and we get $\theta_0(u_1,-u_1)=\theta_0(-u_1,u_1)<0$, consequently. Therefore, we have $(u_1,u_1)\in U_0$ or $(-u_1,u_1)\in U_0$, equivalently. This allows us to assure that $\psi_i$ is well-defined on pairs of antipodal points. This map is digitally continuous on $U_i$ and satisfies $\psi_i(u_1,v_1)=\psi_i(\pm u_1,\pm v_1)$ for $0 \leq i \leq k$ and the induced map $\bar{\psi}_i : U_i /\!\sim \,\to (\dign)^{I_{m_1+m_2}}$ admits an explicit motion planner on the digital projective space $\dign$.

For $2 \leq q \leq r$, we use the following subsets of $S_{\text{min}} ^{n_q} \times S_{\text{min}} ^{n_q}$. The case that $n_q$ is odd, we take subsets
\begin{flalign*}
V_0 &= \left\{ (u_q, v_q) \in S_{\text{min}} ^{n_q} \times S_{\text{min}} ^{n_q} \text{ : } v_q \neq \pm u_q \right\}, \\ 
V_1 &= \left\{ (u_q, v_q) \in S_{\text{min}} ^{n_q} \times S_{\text{min}} ^{n_q} \text{ : } v_q = \pm u_q \right\}.
\end{flalign*}
The case that $n_q$ is even, we consider subsets
\begin{flalign*}
V_0 &= \left\{ (u_q, v_q) \in S_{\text{min}} ^{n_q} \times S_{\text{min}} ^{n_q} \text{ : } v_q \neq \pm u_q \right\}, \\
V_1 &= \left\{ (u_q, v_q) \in S_{\text{min}} ^{n_q} \times S_{\text{min}} ^{n_q} \text{ : } v_q = \pm u_q , u_q \neq \pm a_q \right\}, \\
V_2 &= \left\{ (u_q, v_q) \in S_{\text{min}} ^{n_q} \times S_{\text{min}} ^{n_q} \text{ : } v_q = \pm u_q , u_q = \pm a_q \right\}.
\end{flalign*}
Here, the fixed element $a_q = (0,\ldots,0,1) \in S_{\text{min}} ^{n_q}$ corresponds to the vanishing point of even dimensional spheres.

We describe the motion planner for a digital sphere by the paths :
\begin{itemize}
\item For $(u_q, v_q) \in V_0$ and for $(u_q,u_q)\in V_1 \cup V_2$, we consider the digital path $\rho(u_q, v_q)$.
\end{itemize}

\begin{itemize}
\item For $(u_q , -u_q) \in V_1$, we consider the digital meridian $\sigma: I_{m_2} \to S_{\text{min}} ^{n_q}$ path $\sigma(u_q , -u_q)$ from $u_q$ to $-u_q$ in the positive direction symmetrically, corresponding digital meridian from $u_q$ to $-u_q$.
\end{itemize}

\begin{itemize}
\item For $(a_q , -a_q)$, we fix digital meridian $\sigma_0: I_{m_2} \to S_{\text{min}} ^{n_q}$ path from $a_q$ to $-a_q$ and we set $\sigma_0(-a_q , a_q) = -\sigma_0 (a_q , -a_q)$.
\end{itemize}

We combine these motion planners in the following way:

Given $i \in \{ 0,1, \ldots , k \}$, let $j_q \in \{ 0,1 \}$ when $n_q$ is odd; or $j_q \in \{ 0,1,2 \}$ when $n_q$ is even for $2 \leq q \leq r$. We define the map
 \begin{displaymath}
\psi_{(i,j_2 , \ldots , j_r)} : U_i \times \prod_{q=2} ^r V_{j_q} \to (\digp)^{I_{m_1 + m_2}}
\end{displaymath} by
\begin{displaymath}
\psi_{(i,j_2 , \ldots , j_r)} (u_1,v_1,u_2,v_2, \ldots , u_r , v_r) = \begin{cases}
[\rho(u_1,v_1) , \omega_2 , \ldots , \omega_r], & \theta_i (u_1,v_1) > 0 \\
[\rho(u_1,v_1) , \omega_2^{\prime} , \ldots , \omega_r^{\prime}], & \theta_i (u_1,v_1) < 0,
\end{cases}
\end{displaymath}
where, for $n_q$ odd,

$\omega_q = \begin{cases}
\sigma(u_q, v_q), & \text{ if } v_q= -u_q \\
\rho (u_q, v_q), & \text{ otherwise, } 
\end{cases}$

$\omega_q^{\prime} = \begin{cases}
\sigma(-u_q, v_q), & \text{ if } v_q= u_q \\
\rho (-u_q, v_q), & \text{ otherwise, } 
\end{cases}$

and, for $n_q$ even,

$\omega_q = \begin{cases}
\sigma (u_q, v_q), & \text{ if } v_q= -u_q , u_q \neq \pm a_q \\
\sigma_0 (u_q, v_q), & \text{ if } v_q= -u_q , u_q = \pm a_q \\
\rho (u_q, v_q), & \text{ otherwise, } 
\end{cases}$

$\omega_q^{\prime} = \begin{cases}
\sigma (-u_q, v_q), & \text{ if } v_q = u_q , u_q \neq \pm a_q \\
\sigma_0 (-u_q, v_q), & \text{ if } v_q = u_q , u_q = \pm a_q \\
\rho (-u_q, v_q), & \text{ otherwise. } 
\end{cases}$

The map is the well-defined digital continuous on $U_i \times \prod_{q=2} ^r V_{j_q}$ . Moreover, the compatibility of this map does not conflict with the equivalence relation (\ref{relation}). 

For $i \in \{0,...,k\}$, $j_q \in \{ 0,1 \}$ when $n_q$ is odd; or $j_q \in \{ 0,1,2 \}$ when $n_q$ is even, $2 \leq q \leq r$ we obtain a digital continuous map 
\begin{displaymath}
\bar{\psi}_{(i , j_2 , \ldots , j_r)}: \dfrac{ U_i \times \prod_{q=2} ^r V_{j_q}}{\sim} \to (\digp)^{I_{m_1 + m_2}}
\end{displaymath}
that satisfies $\bar{\psi}_{(i,j_2 , \ldots , j_q)} (u_1,v_1,u_2,v_2, \ldots , u_r , v_r) = [\rho(u_1,v_1) , \bar{\omega}_2 , \ldots , \bar{\omega}_r]$, where for $n_q$ odd,

$\bar{\omega}_q = \begin{cases}
\sigma(u_q, v_q), & \text{ if } v_q= -u_q \\
\rho (u_q, v_q), & \text{ otherwise, } 
\end{cases}$

and, for $n_q$ even,

$\bar{\omega}_q= \begin{cases}
\sigma (u_q, v_q), & \text{ if } v_q= -u_q , u_q \neq \pm a_q \\
\sigma_0 (u_q, v_q), & \text{ if } v_q= -u_q , u_q = \pm a_q \\
\rho (u_q, v_q), & \text{ otherwise, } 
\end{cases}$

For $i \in \{ 0,1, \ldots , k \}$, $j_q \in \{0,1\}$ when $n_q$ is odd or; $j_q \in \{0,1,2\}$ when $n_q$ is even, $2 \leq q \leq r$,
\begin{displaymath}
W_l = \bigcup_{i+\sum_{q=2}^r j_q =l} \left( U_i \times \prod_{q=2}^r  V_q \right) \subset  (S_{\text{min}}^{n_1} \times S_{\text{min}}^{n_1}) \times  \cdots \times (S_{\text{min}}^{n_r} \times S_{\text{min}}^{n_r}),
\end{displaymath}
where $l=0, \ldots ,k+ \sum_{q=2}^r \digtc (S_{\text{min}}^{n_q})=\digtc(\dign) + \sum_{q=2}^r \digtc (S_{\text{min}}^{n_q})$ which is the disjoint union. $S_{\text{min}}^{n_1} \times S_{\text{min}}^{n_1} \times S_{\text{min}}^{n_2} \times S_{\text{min}}^{n_2} \times \cdots \times S_{\text{min}}^{n_r} \times S_{\text{min}}^{n_r} \cong \digs \times \digs$ contains all the subsets $W_l$ concerning the relation (\ref{relation}).

In the progression to the quotient space,
\begin{displaymath}
\bar{W}_{l} = \bigcup_{i+\sum_{q=2}^r j_q =l} \dfrac {(U_i \times \prod_{q=2}^r V_q)} {\sim} \subset \digp \times \digp
\end{displaymath}
is a disjoint union where $l=0, \ldots , k + \sum_{q=2}^r \digtc (S_{\text{min}}^{n_q})$. The obtained maps $\bar{\psi}_{(i , j_2 , \ldots , j_r)}$ we get continuous explicitly motion planner on $\bar{W}_{l}$ and $\bigcup_{l=0} ^{k + \sum_{q=2}^r \digtc (S_{\text{min}}^{n_q})} \bar{W}_{l}$ covers on $\digp \times \digp$. Hence, we conclude that
\begin{displaymath}
\digtc (\digp) \leq k+ \sum_{q=2}^{r} \digtc (S_{\text{min}}^{n_q}) = \digtc (\dign) + \digtc (S_{\text{min}}^{n_q}).
\end{displaymath}
\end{proof}

\begin{ex} We analyze the digital topological complexity of digital projective product space for $n_1=2$, $n_2=2$ and $r=2$. We state that 
\begin{displaymath}
\digtc(\text{d-}P) = \digtc \left( \dfrac{S_{\text{min}} ^{2} \times S_{\text{min}} ^{2}}{\sim} \right) \leq \digtc (\text{d-}P^{2}) + \digtc (S_{\text{min}} ^{2}).
\end{displaymath}

\begin{figure}[H]
\centering
\scalebox{0.08}{\includegraphics{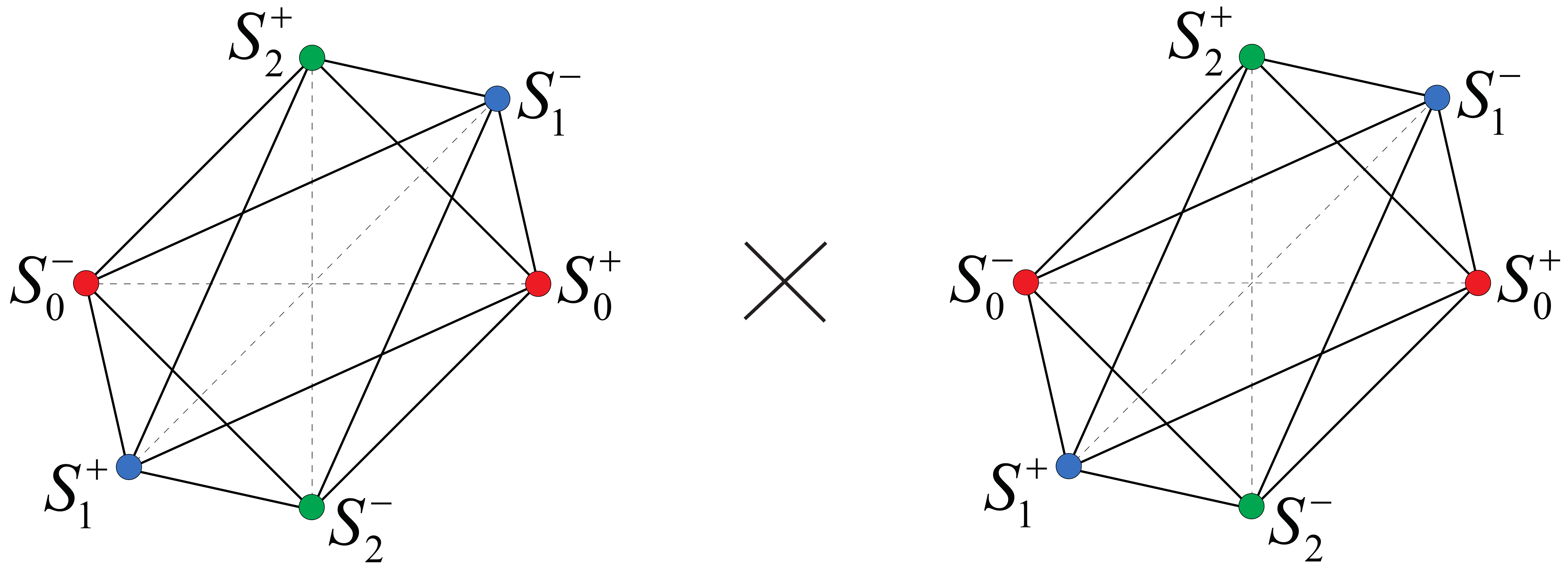}}
\caption{$S_{\text{min}} ^{2} \times S_{\text{min}} ^{2}$}
\label{sp2}
\end{figure}
We set explicitly the motion planner for the digital projective space $d$-$P^{2}=S_{\text{min}} ^{2}/\sim$ by using the characterization of digital non-singular maps that we defined and for digital sphere $S_{\text{min}} ^{2}$ that we have from the inspiration of \cite{FTY} and \cite{Far}, respectively.

We find a cover of $S_{\text{min}} ^{2} \times S_{\text{min}} ^{2} \subset \mathbb{Z}^3 \times \mathbb{Z}^3$ by processing similarly as in \cite{FTY}. The digital non-singular map $\mathbb{Z}^4 \times \mathbb{Z}^4 \rightarrow \mathbb{Z}^4$ has a restriction onto $\mathbb{Z}^3 \subset \mathbb{Z}^4$ and this provides us the digital non-singular $\theta: \mathbb{Z}^3 \times \mathbb{Z}^3 \rightarrow \mathbb{Z}^4$ with the formula 
\begin{displaymath}
\theta(u_1,v_1) =\langle u_1, v_1\rangle -
\begin{vmatrix}
u_{1_1} & u_{1_2}\\
v_{1_1} & v_{1_2}
\end{vmatrix}
i 
-
\begin{vmatrix}
u_{1_1} & u_{1_3}\\
v_{1_1} & v_{1_3}
\end{vmatrix}
j
-
\begin{vmatrix}
u_{1_2} & u_{1_3}\\
v_{1_2} & v_{1_3}
\end{vmatrix}
k,
\end{displaymath}
where $u_1=(u_{1_1}, u_{1_2}, u_{1_3}), v_1=(v_{1_1}, v_{1_2}, v_{1_3}) \in \mathbb{Z}^3$, the imaginary units $i, j, k \in \mathbb{Z}^4$, $\langle u_1, v_1 \rangle$ represents the scalar product of $u_1$ and $v_1$ and 
\begin{displaymath}
\theta(u_1,v_1) = \theta_2(u_1,v_1) + \theta_1(u_1,v_1)i + \theta_2(u_1,v_1)j + \theta_3(u_1,v_1)k.
\end{displaymath}
 We indicate that subsets are compatible with the antipodal relation on $S_{\text{min}} ^{2}$:
\begin{flalign*}
U_0 &= \left\{ (u_1,v_1) \in S_{\text{min}} ^{2} \times S_{\text{min}} ^{2} \text{ : } \theta_2 (u_1,v_1) \neq 0 \right\}, \\
U_1 &= \left\{ ((u_1,v_1) \in S_{\text{min}} ^{2} \times S_{\text{min}} ^{2} \text{ : } \theta_2 (u_1,v_1) = 0, \theta_1(u_1,v_1) \neq 0 \right\}, \\
U_2 &= \left\{ (u_1,v_1) \in S_{\text{min}} ^{2} \times S_{\text{min}} ^{2} \text{ : } \theta_2 (u_1,v_1) = 0, \theta_1(u_1,v_1) = 0, \theta_2(u_1,v_1) \neq 0 \right\}  \\
U_3 &= \left\{ (u_1,v_1) \in S_{\text{min}} ^{2} \times S_{\text{min}} ^{2} \text{ : } \theta_2 (u_1,v_1) = 0, \theta_1(u_1,v_1) = 0, \theta_2(u_1,v_1) = 0, \theta_3(u_1,v_1) \neq 0 \right\}.
\end{flalign*} 

Notice that $\bigcup_{i=0} ^3 U_i$ includes the disjoint subsets of $S_{\text{min}} ^{2} \times S_{\text{min}} ^{2}$ and that provides us a cover of $S_{\text{min}} ^{2} \times S_{\text{min}} ^{2}$.

We consider the digital path $\rho(A,B): I_{m_1} \to S_{\text{min}} ^n$ with $\rho(-A , -B) = -\rho(A,B)$ from $A$ to $B$  for except antipodal points $A,B$ $(A \neq -B)$ of digital sphere $S_{\text{min}} ^n$. Remind that $\rho(-A , -B) = -\rho(A,B)$ and that $\rho(A,A)$ is the constant digital path.

We set the map, for $0 \leq i \leq 3$, $\psi_i: U_i \to (\text{d-}P^2)^{I_{m_1+m_2}}$ by
\begin{displaymath}
\psi_i (u_1,v_1) = \begin{cases} 
[\rho(u_1,v_1)], &\text{ if } \theta_i (u_1,v_1) > 0 \\
[\rho(-u_1,v_1], &\text{ if } \theta_i (u_1,v_1) < 0.
\end{cases}
\end{displaymath}
For any $u_1 \in S_{\text{min}} ^2$, we have $\theta_2 (u_1,u_1) > 0$ and $\theta_2(u_1,-u_1)=\theta_2(-u_1,u_1)<0$. Accordingly, we have a pair $(u_1,u_1)\in U_0$ or $(-u_1,u_1)\in U_0$ for guaranteeing that $\psi_i$ is well-defined on pairs of antipodal points. This map is digitally continuous on $U_i$ and satisfies $\psi_i(u_1,v_1)=\psi_i(\pm u_1,\pm v_1)$ for $0 \leq i \leq k$ and the induced map $\bar{\psi}_i : U_i /\!\sim \to (\text{d-}P^2)^{I_{m_1+m_2}}$ gives us an explicit motion planner on $d$-$P^{2}$.

We use the following subsets of $S_{\text{min}} ^{2} \times S_{\text{min}} ^{2}$. The case that $n_q$ is even, we consider:
\begin{flalign*}
V_0 &= \left\{ (u_2, v_2) \in S_{\text{min}} ^{2} \times S_{\text{min}} ^{2} \text{ : } v_2 \neq \pm u_2 \right\} \\
V_1 &= \left\{ (u_2, v_2) \in S_{\text{min}} ^{2} \times S_{\text{min}} ^{2} \text{ : } v_2 = \pm u_2 , u_2 \neq \pm a_2 \right\}  \\
V_2 &= \left\{ (u_2, v_2) \in S_{\text{min}} ^{2} \times S_{\text{min}} ^{2} \text{ : } v_2 = \pm u_2 , u_2 = \pm a_2 \right\}.
\end{flalign*}
Here we take $a_2 = (0,0,1) \in S_{\text{min}} ^{2}$ corresponds to the vanishing point of even dimensional spheres.

We present the motion planner for even dimensional digital sphere as before:
\begin{itemize}
\item For $(u_2, v_2) \in V_0$ and for $(u_2, v_2)\in V_1 \cup V_2$, we consider the digital path $\rho(u_2, v_2)$.
\end{itemize}

\begin{itemize}
\item For $(u_2 , -u_2) \in V_1$, we use the digital meridian $\sigma: I_{m_2} \to S_{\text{min}} ^{2}$ path $\sigma(u_2 , -u_2)$ from $u_2$ to $-u_2$ in the positive direction symmetrically.
\end{itemize}

\begin{itemize}
\item For $(a_2 , -a_2)$, we fix a meridian $\sigma_0: I_{m_2} \to S_{\text{min}} ^{2}$ digital path from $a_2$ to $-a_2$ and we set $\sigma_0(-a_2 , a_2) = -\sigma_0 (a_2 , -a_2)$.
\end{itemize}

We assemble these motion planners on $U_i \times V_j \subset S_{\text{min}} ^2 \times S_{\text{min}} ^2 \times S_{\text{min}} ^2 \times S_{\text{min}} ^2$ for $j \in \{ 0,1,2 \}$ in the following way:

\underline{The motion planner on $U_i \times V_0$:} We define the map $\psi_{(i,0)}:  U_i \times V_0 \to (\text{d-}P)^{I_{m_1+m_2}}$ for $i \in \{ 0,1,2,3 \}$ by
\begin{displaymath}
\psi_{(i,0)} (u_1,v_1,u_2,v_2) =
\begin{cases}
[ \rho(u_1,v_1) , \rho(u_2,v_2) ] & \text{ if } \theta_i (u_1,v_1) > 0, \\
[ \rho(-u_1,v_1) , \rho(-u_2,v_2) ] & \text{ if } \theta_i (u_1,v_1) < 0,
\end{cases}
\end{displaymath}
where $(u_1,v_1) \in U_i$ and $(u_2,v_2) \in V_0$. As in the proof of Theorem 3.7., we consider the image of $(u_1,v_1,u_2,v_2)$ under the isomorphism and specify the image of $\psi_{(i,0)}$ as below.
\begin{itemize}[label=$\ast$]
\item If $(u_1,u_2,v_1,v_2) \in U_i \times V_0$ with $\theta_i (u_1,v_1) >0$ and $v_2 \neq \pm u_2$, then 
\begin{displaymath}
\psi_{(i,0)} (u_1,u_2,v_1,v_2) = [\rho(u_1,v_1) , \rho(u_2,v_2) ].
\end{displaymath}
\end{itemize}

\begin{itemize}[label=$\ast$]
\item For $(u_1,u_2,-v_1,-v_2) \in U_i \times V_0$, we have $\theta_i (u_1,v_1) <0$ and $v_2 \neq \pm u_2$. So, we get
\begin{displaymath}
\psi_{(i,0)} (u_1,u_2,-v_1,-v_2) = [\rho(-u_1,-v_1) , \rho(-u_2,-v_2) ] = [\rho(u_1,v_1) , \rho(u_2,v_2) ].
\end{displaymath}
\end{itemize}

\begin{itemize}[label=$\ast$]
\item For $(-u_1,-u_2,v_1,v_2) \in U_i \times V_0$, we have $\theta_i (u_1,v_1) <0$ and $v_2 \neq \pm u_2$. In that case, we have
\begin{displaymath}
\psi_{(i,0)} (-u_1,-u_2,v_1,v_2) = [\rho(u_1,v_1) , \rho(u_2,v_2) ].
\end{displaymath}
\end{itemize}

\begin{itemize}[label=$\ast$]
\item For $(-u_1,-u_2,-v_1,-v_2) \in U_i \times V_0$, we have $\theta_i (u_1,v_1) >0$ and $v_2 \neq \pm u_2$. Thus, we obtain
\begin{displaymath}
\psi_{(i,0)} (-u_1,-u_2,-v_1,-v_2) = [\rho(-u_1,-v_1) , \rho(-u_2,-v_2) ] = [\rho(u_1,v_1) , \rho(u_2,v_2) ].
\end{displaymath}
\end{itemize}

\underline{The motion planner on $U_i \times V_1$:} We define the map $\psi_{(i,1)}:  U_i \times V_1 \to (\text{d-}P)^{I_{m_1+m_2}}$ for $i \in \{ 0,1,2,3 \}$ by
\begin{displaymath}
\psi_{(i,1)} (u_1,v_1,u_2,v_2) =
\begin{cases}
[ \rho(u_1,v_1) , \sigma(u_2,v_2) ] & \text{ if } \theta_i (u_1,v_1) > 0, v_2 = -u_2 , u_2 \neq \pm a_2 \\
[ \rho(u_1,v_1) , \rho(u_2,v_2) ] & \text{ if } \theta_i (u_1,v_1) > 0, v_2 = u_2 , u_2 \neq \pm a_2 \\
[ \rho(-u_1,v_1) , \rho(-u_2,v_2) ] & \text{ if } \theta_i (u_1,v_1) < 0, v_2 = -u_2 , u_2 \neq \pm a_2 \\
[ \rho(-u_1,v_1) , \sigma(-u_2,v_2) ] & \text{ if } \theta_i (u_1,v_1) < 0, v_2 = u_2 , u_2 \neq \pm a_2
\end{cases}
\end{displaymath}

\begin{itemize}[label=$\ast$]
\item Let $(u_1,u_2,v_1,v_2) \in U_i \times V_1$ with $\theta_i (u_1,v_1) >0$, $v_2 = -u_2$ and $u_2 \neq \pm a_2$. After that, we get
\begin{displaymath}
\psi_{(i,1)} (u_1,u_2,v_1,v_2) = \psi_{(i,1)} (u_1,u_2,v_1,-u_2) = [\rho(u_1,v_1) , \sigma(u_2,-u_2) ].
\end{displaymath}
\end{itemize}

\begin{itemize}[label=$\ast$]
\item For $(u_1,u_2,-v_1,-v_2) \in U_i \times V_1$, we have $\theta_i (u_1,v_1) <0$, $v_2 = -u_2$ and $u_2 \neq \pm a_2$. Thus, we have
\begin{flalign*}
\psi_{(i,1)} (u_1,u_2,-v_1,-v_2) = \psi_{(i,1)} (u_1,u_2,-v_1,u_2) &= [\rho(-u_1,-v_1) , \sigma(-u_2,u_2) ] \\ &= [\rho(u_1,v_1) , \sigma(u_2,-u_2) ].
\end{flalign*}
\end{itemize}

\begin{itemize}[label=$\ast$]
\item For $(-u_1,-u_2,v_1,v_2) \in U_i \times V_1$, we have $\theta_i (u_1,v_1) <0$, $v_2 = -u_2$ and $u_2 \neq \pm a_2$. Hence, we obtain 
\begin{displaymath}
\psi_{(i,1)} (-u_1,-u_2,v_1,v_2) = \psi_{(i,1)} (-u_1,-u_2,v_1,-u_2) = [\rho(u_1,v_1) , \sigma(u_2,-u_2) ].
\end{displaymath}
\end{itemize}

\begin{itemize}[label=$\ast$]
\item For $(-u_1,-u_2,-v_1,-v_2) \in U_i \times V_1$, we have $\theta_i (u_1,v_1) >0$, $v_2 = -u_2$ and $u_2 \neq \pm a_2$. Thereafter, we acquire
\begin{flalign*}
\psi_{(i,1)} (-u_1,-u_2,-v_1,-v_2) = \psi_{(i,1)} (-u_1,-u_2,-v_1,u_2) &= [\rho(-u_1,-v_1) , \sigma(-u_2,u_2) ] \\ &= [\rho(u_1,v_1) , \sigma(u_2,-u_2) ].
\end{flalign*}
\end{itemize}

\begin{itemize}[label=$\ast$]
\item Let $(u_1,u_2,v_1,v_2) \in U_i \times V_1$ with $\theta_i (u_1,v_1) >0$, $v_2 = u_2$ and $u_2 \neq \pm a_2$. So, there exists
\begin{displaymath}
\psi_{(i,1)} (u_1,u_2,v_1,v_2) = \psi_{(i,1)} (u_1,u_2,v_1,u_2) = [\rho(u_1,v_1) , \rho(u_2,u_2) ].
\end{displaymath}
\end{itemize}

\begin{itemize}[label=$\ast$]
\item For $(u_1,u_2,-v-1,-v_2) \in U_i \times V_1$, we have $\theta_i (u_1,v_1) <0$, $v_2 = u_2$ and $u_2 \neq \pm a_2$. At that case, this satisfies 
\begin{flalign*}
\psi_{(i,1)} (u_1,u_2,-v_1,-v_2) = \psi_{(i,1)} (u_1,u_2,-v_1,-u_2) &= [\rho(-u_1,-v-1) , \rho(-u_2,-u_2) ] \\ &= [\rho(u_1,v_1) , \rho(u_2,u_2) ].
\end{flalign*}
\end{itemize}

\begin{itemize}[label=$\ast$]
\item For $(-u_1,-u_2,v_1,v_2) \in U_i \times V_1$, we have $\theta_i (u_1,v_1) <0$, $v_2 = u_2$ and $u_2 \neq \pm a_2$. Afterwards, this gives that
\begin{displaymath}
\psi_{(i,1)} (-u_1,-u_2,v_1,v_2) = \psi_{(i,1)} (-u_1,-u_2,v_1,u_2) = [\rho(u_1,v_1) , \rho(u_2,u_2) ].
\end{displaymath}
\end{itemize}

\begin{itemize}[label=$\ast$]
\item For $(-u_1,-u_2,-v_1,-v_2) \in U_i \times V_1$, we have $\theta_i (u_1,v_1) >0$, $v_2 = u_2$ and $u_2 \neq \pm a_2$. Therefore, this provides that
\begin{flalign*}
\psi_{(i,1)} (-u_1,-u_2,-v_1,-v_2) = \psi_{(i,1)} (-u_1,-u_2,-v_1,-u_2) &= [\rho(-u_1,-v_1) , \rho(-u_2,-u_2) ] \\ &= [\rho(u_1,v_1) , \rho(u_2,u_2) ].
\end{flalign*}
\end{itemize}

\underline{The motion planner on $U_i \times V_2$:} We define the map $\psi_{(i,2)}:  U_i \times V_2 \to (\text{d-}P)^{I_{m_1+m_2}}$ for $i \in \{ 0,1,2,3 \}$ by
\begin{displaymath}
\psi_{(i,2)} (u_1,v_1,u_2,v_2) =
\begin{cases}
[ \rho(u_1,v_1) , \sigma_0 (u_2,v_2) ] & \text{ if } \theta_i (u_1,v_1) > 0, v_2 = -u_2 , u_2 = \pm a_2 \\
[ \rho(u_1,v_1) , \rho(u_2,v_2) ] & \text{ if } \theta_i (u_1,v_1) > 0, v_2 = u_2 , u_2 = \pm a_2 \\
[ \rho(-u_1,v_1) , \rho(-u_2,v_2) ] & \text{ if } \theta_i (u_1,v_1) < 0, v_2 = -u_2 , u_2 = \pm a_2 \\
[ \rho(-u_1,v_1) , \sigma_0 (-u_2,v_2) ] & \text{ if } \theta_i (u_1,v_1) < 0, v_2 = u_2 , u_2 = \pm a_2.
\end{cases}
\end{displaymath}

\begin{itemize}[label=$\ast$]
\item Let $(u_1,u_2,v_1,v_2) \in U_i \times V_2$ with $\theta_i (u_1,v_1) >0$, $v_2 = -u_2$ and $u_2 = \pm a_2$. Then we have
\begin{displaymath}
\psi_{(i,2)} (u_1,u_2,v_1,v_2) = \psi_{(i,2)} (u_1,\pm a_2,v_1,-u_2) = [\rho(u_1,v_1) , \sigma_0 (\pm a_2,\mp a_2) ].
\end{displaymath}
\end{itemize}

\begin{itemize}[label=$\ast$]
\item For $(u_1,u_2,-v_1,-v_2) \in U_i \times V_2$, we have $\theta_i (u_1,v_1) <0$, $v_2 = -u_2$ and $u_2 = \pm a_2$. Hence, we obtain
\begin{flalign*}
\psi_{(i,2)} (u_1,u_2,-v_1,-v_2) = \psi_{(i,2)} (u_1,\pm a_2,-v_1,\pm a_2) &= [\rho(-u_1,-v_1) , \sigma_0 (\mp a_2,\pm a_2) ] \\ &= [\rho(u_1,v_1) , \sigma_0 (\pm a_2,\mp a_2) ].
\end{flalign*}
\end{itemize}

\begin{itemize}[label=$\ast$]
\item For $(-u_1,-u_2,v_1,v_2) \in U_i \times V_2$, we have $\theta_i (u_1,v_1) <0$, $v_2 = -u_2$ and $u_2 = \pm a_2$. So, we get
\begin{displaymath}
\psi_{(i,2)} (-u_1,-u_2,v_1,v_2) = \psi_{(i,2)} (-u_1,\mp a_2,v_1,\mp a_2) = [\rho(u_1,v_1) , \sigma_0 (\pm a_2,\mp a_2) ].
\end{displaymath}
\end{itemize}

\begin{itemize}[label=$\ast$]
\item For $(-u_1,-u_2,-v_1,-v_2) \in U_i \times V_2$, we have $\theta_i (u_1,v_1) >0$, $v_2 = -u_2$ and $u_2 = \pm a_2$. Thus, we state that
\begin{flalign*}
\psi_{(i,2)} (-u_1,-u_2,-v_1,-v_2) &= \psi_{(i,2)} (-u_1,\mp a_2,-v_1,\pm a_2) \\ &= [\rho(-u_1,-v_1) , \sigma_0 (\mp a_2,\pm a_2) ] \\ &= [\rho(u_1,v_1) , \sigma_0 (\pm a_2,\mp a_2) ].
\end{flalign*}
\end{itemize}

\begin{itemize}[label=$\ast$]
\item Let $(u_1,u_2,v_1,v_2) \in U_i \times V_2$ with $\theta_i (u_1, v_1) >0$, $v_2 = u_2$ and $u_2 = \pm a_2$. Then
\begin{displaymath}
\psi_{(i,2)} (u_1,u_2,v_1,v_2) = \psi_{(i,2)} (u_1,\pm a_2,v_1,\pm a_2) = [\rho(u_1,v_1) , \rho(\pm a_2,\pm a_2) ].
\end{displaymath}
\end{itemize}

\begin{itemize}[label=$\ast$]
\item For $(u_1,u_2,-v_1,-v_2) \in U_i \times V_2$, we have $\theta_i (u_1,v_1) <0$, $v_2 = u_2$ and $u_2 = \pm a_2$. Next, this gives that
\begin{flalign*}
\psi_{(i,2)} (u_1,u_2,-v_1,-v_2) = \psi_{(i,2)} (u_1,\pm a_2,-v_1,\mp a_2) &= [\rho(-u_1,-v_1) , \rho(\mp a_2,\mp a_2) ] \\ &= [\rho(u_1,v_1) , \rho(\pm a_2,\pm a_2) ].
\end{flalign*}
\end{itemize}

\begin{itemize}[label=$\ast$]
\item For $(-u_1,-u_2,v_1,v_2) \in U_i \times V_2$, we have $\theta_i (u_1,v_1) <0$, $v_2 = u_2$ and $u_2  = \pm a_2$. Hence, this provides that
\begin{displaymath}
\psi_{(i,2)} (-u_1,-u_2,v_1,v_2) = \psi_{(i,2)} (-u_1,\mp a_2,v_1,\pm a_2) = [\rho(u_1,v_1) , \rho(\pm a_2,\pm a_2) ].
\end{displaymath}
\end{itemize}

\begin{itemize}[label=$\ast$]
\item For $(-u_1,-u_2,-v_1,-v_2) \in U_i \times V_2$, we have $\theta_i (u_1,v_1) >0$, $v_2 = u_2$ and $u_2  = \pm a_2$. Accordingly, this satisfies 
\begin{flalign*}
\psi_{(i,2)} (-u_1,-u_2,-v_1,-v_2) &= \psi_{(i,2)} (-u_1,\mp a_2,-v_1,\mp a_2) \\ &= [\rho(-u_1,-v_1) , \rho(\mp a_2,\mp a_2) ] \\ &= [\rho(u_1,v_1) , \rho(\pm a_2,\pm a_2) ].
\end{flalign*}
\end{itemize}

These constructions yield the following maps by considering the quotient space since the equivalence classes are equal.

\begin{itemize}
\item $\bar{\psi}_{(i,0)} : \dfrac{\left (U_i \times V_0 \right)}{\sim} \to (\text{d-}P)^{I_{m_1+m_2}}$ is defined by
\begin{displaymath}
\bar{\psi}_{(i,0)} ([u_1,v_1,u_2,v_2]) = [\rho(u_1,v_1) , \rho(u_2,v_2) ]
\end{displaymath}
\end{itemize}

\begin{itemize}
\item $\bar{\psi}_{(i,1)} : \dfrac{\left (U_i \times V_1 \right)}{\sim} \to (\text{d-}P)^{I_{m_1+m_2}}$ is given by
\begin{displaymath}
\bar{\psi}_{(i,1)} ([u_1,v_1,u_2,v_2]) = 
\begin{cases}
[\rho(u_1,v_1) , \sigma(u_2,v_2)] & \text{ if } \theta_i (u_1,v_1) \neq 0 , v_2 = -u_2 , u_2 \neq \pm a_2 \\
[\rho(u_1,v_1) , \rho(u_2,v_2)] & \text{ if } \theta_i (u_1,v_1) \neq 0 , v_2 = u_2 , u_2 \neq \pm a_2
\end{cases}
\end{displaymath}
\end{itemize}

\begin{itemize}
\item $\bar{\psi}_{(i,2)} : \dfrac{\left (U_i \times V_2 \right)}{\sim} \to (\text{d-}P)^{I_{m_1+m_2}}$ is set by
\begin{displaymath}
\bar{\psi}_{(i,2)} ([u_1,v_1,u_2,v_2]) = 
\begin{cases}
[\rho(u_1,v_1) , \sigma_0 (u_2,v_2)] & \text{ if } \theta_i (u_1,v_1) \neq 0 , v_2 = -u_2 , u_2 \neq \pm a_2 \\
[\rho(u_1,v_1) , \rho(u_2,v_2)] & \text{ if } \theta_i (u_1,v_1) \neq 0 , v_2 = u_2 , u_2 \neq \pm a_2,
\end{cases}
\end{displaymath}
\end{itemize}
where $i \in \{ 0,1,2,3 \}$ and $\text{d-}P = \dfrac{S_{\text{min}} ^{2} \times S_{\text{min}} ^{2}}{\sim}$.

We obtain an explicit construction of $\dfrac{S_{\text{min}} ^{2} \times S_{\text{min}} ^{2}}{\sim} = \dfrac{U_0}{\sim} \cup \dfrac{U_1}{\sim} \cup \dfrac{U_2}{\sim} \cup \dfrac{U_3}{\sim}$ and $S_{\text{min}} ^{2} \times S_{\text{min}} ^{2} = V_0 \cup V_1 \cup V_2$. 
We set 
\begin{displaymath}
W_l = \bigcup_{i+j = l} (U_i \times V_j) \subset S_{\text{min}} ^{2} \times S_{\text{min}} ^{2} \times S_{\text{min}} ^{2} \times S_{\text{min}} ^{2},
\end{displaymath}
 where $i \in \{ 0,1,2,3 \}$ and $j \in \{ 0,1,2 \}$ that is the disjoint union. All subsets of $S_{\text{min}} ^{2} \times S_{\text{min}} ^{2} \times S_{\text{min}} ^{2} \times S_{\text{min}} ^{2}$ are compatible with respect to the equivalence relation. In the quotient space, 
\begin{displaymath}
\bar{W}_l = \bigcup_{i+j = l} \dfrac{U_i \times V_j}{\sim} \subset \text{d-}P \times \text{d-}P
\end{displaymath}
 is the disjoint union, where $l = 0, \ldots , 3 + \digtc{S_{\text{min}} ^{2}}$. We describe a motion planning strategy over each $\bar{W}_l$ and $\bigcup_{l = 0} ^5 \bar{W}_l$ is a cover of $\text{d-}P \times \text{d-}P$. Therefore, we conclude that
\begin{displaymath}
\digtc (\text{d-}P) \leq \digtc (\text{d-}P^2) + \digtc (S_{\text{min}} ^{2}).
\end{displaymath}
\end{ex}

\section{Conclusion}
The combination of topological structures with robotics formed a new area called topological robotics. Although robotics is a practical discipline, there is a theoretical side of the subject. The theoretical idea of robotics has been associated with many branches of mathematics. Topology has played a key role in implementing great ideas. For instance, scholars have discussed the topological problems inspired by robotics and studied motion planning problem, as well as the concept of Farber's topological complexity in detail. When the digital topological tools more specifically the notion of the digital topological complexity and related invariants are utilized in finding solutions to problems, interdisciplinary interaction will increase and hence this will open new windows in the field.

In this paper, we aim to introduce the digital projective product spaces and generalize the digital projective spaces by using digital spheres~\cite{Evako}. The main goal is to deal with the digital topological complexity and digital LS-category of these spaces. We begin with determining an upper bound for the digital LS-category and ultimately an upper bound for the digital topological complexity of the digital projective product spaces. Additionally, we define the digital non-singular map, and we use the digital non-singular map characterization inspired by~\cite{FTY} to measure the digital topological complexity of the digital projective spaces. We prove the relation between the digital topological complexity of the digital projective product spaces and the sum of the digital topological complexity of the digital projective space associated with the first digital sphere and the digital topological complexity of the remaining digital spheres. We accomplish this by constructing an explicit motion planning on these spaces. In this context, the advantages of more direct methods in the digital sense provide the results in ~\cite{FV} apart from requiring cohomological operational lower bound properties. In particular, we give examples on specific spaces to clarify our results.

This leads us to work on the digital higher topological complexity and related invariants of the digital projective product spaces, which is an open problem.
\section{Acknowledgment} The Scientific and Technological Research Council of Turkey TUBITAK-2211/A grants the first author as a fellowship.
\bibliographystyle{plain}

\end{document}